\documentclass[letterpaper, 10 pt, conference]{ieeeconf}  

\IEEEoverridecommandlockouts                              

\usepackage{graphicx}
\usepackage{graphics} 
\usepackage{times} 
\usepackage{amsmath} 
\usepackage{amssymb}  
\usepackage{color}
\let\labelindent\relax
\usepackage{enumitem}
\usepackage{cite}

\newtheorem{theorem}{\bf Theorem}

\newtheorem{definition}{\bf Definition}
\newtheorem{claim}{\bf Claim}

\newtheorem{lemma}{\bf Lemma}

\def\QED{~\rule[-1pt]{5pt}{5pt}\par\medskip}

\newcommand{\bx}{{\bf x}}
\newcommand{\by}{{\bf y}}

\newcommand{\bu}{{\bf u}}
\newcommand{\bv}{{\bf v}}
\newcommand{\bw}{{\bf w}}

\title{\LARGE \bf
SDP-based Joint Sensor and Controller Design for \\
Information-regularized Optimal LQG Control
}

\author{Takashi Tanaka$^{1}$ \and Henrik Sandberg$^{2}$ 
\thanks{$^{1}$Laboratory for Information and Decision Systems, Massachusetts Institute of Technology
        {\tt\small ttanaka@mit.edu}. This author is supported by the JSPS Postdoctoral Fellowship.}
\thanks{$^{2}$Department of Automatic Control,
School of Electrical Engineering, KTH Royal Institute of Technology
        {\tt\small hsan@kth.se}
} %
}

\begin{document}

\maketitle
\thispagestyle{empty}
\pagestyle{empty}

\begin{abstract}
We consider a joint sensor and controller design problem for linear Gaussian stochastic systems in which a weighted sum of quadratic control cost and the amount of information acquired by the sensor is minimized. 
This problem formulation is motivated by situations where a control law must be designed in the presence of sensing, communication, and privacy constraints.
We show that the optimal joint sensor-controller design is relatively easy when the sensing policy is restricted to be linear. 
Namely, an explicit form of the optimal linear sensor equation, the Kalman filter, and the certainty equivalence controller that jointly solves the problem can be efficiently found by semidefinite programming (SDP). 
Whether the linearity assumption in our design is restrictive or not is currently an open problem.
\end{abstract}

\section*{Notation}
Lower-case bold characters such as $\bx$ are used to represent random variables. By $\bx\sim \mathcal{N}(\mu, \Sigma)$, we mean that $\bx$ is a multi-dimensional Gaussian random variable with mean vector $\mu$ and covariance matrix $\Sigma$. If $\bx_1, \bx_2, \cdots$ is a sequence of random variables, we write $\bx^t\triangleq (\bx_1,\cdots, \bx_t)$. 
Let $\mathbb{S}_{++}^n$ (resp. $\mathbb{S}_+^n$) be the space of $n$-dimensional real-valued symmetric positive definite (resp. positive semidefinite) matrices. A condition $M\in \mathbb{S}_{++}^n$ (resp. $M\in \mathbb{S}_+^n$) is also written as $M \succ 0$ (resp. $M \succeq 0$).
For a real-valued vector $x\in\mathbb{R}^n$ and a positive semidefinite matrix $Q\succeq 0$, we write $\|x\|_Q^2=x^\top Q x$.

\section{Introduction}
The classical LQG control theory is not concerned with the information-theoretic cost of communication between the sensor and controller devices.
However, communication could be a costly process in practice due to various reasons.
Motivated by such situations, in this paper, we consider a joint sensor and controller design problem, aiming at minimizing the communication between these devices.

In Fig.~\ref{fig:infoLQG}, the dynamical system block represents a linear stochastic system
\begin{equation}
\label{eqprocess}
\bx_{t+1}=A_t \bx_t + B_t \bu_t + \bw_t, \; t=1,\cdots, T
\end{equation}
where $\bx_1\sim \mathcal{N}(0,P_{1|0})$ and $\bw_t\sim \mathcal{N}(0, W_t), t=1,\cdots, T$ are independent random vectors. We assume $P_{1|0}\succ 0$ and $W_t \succ 0$ for every $t=1,\cdots,T$. Suppose $\bx_t\in \mathbb{R}^{n_t}, \bu_t\in \mathbb{R}^{m_t}$, and dimensions can be time varying. The \emph{sensor} block is a data processing unit that has an access to the entire history of the state variables $\bx^t$, the history of control inputs $\bu^{t-1}$ and signals it has generated in the past $\by^{t-1}$, and generates a signal $\by_t \in \mathbb{R}^{r_t}$ at time step $t$. We denote by $\pi_s$ the space of sensor's policies, whose mathematical description will be specified shortly.
The \emph{controller} block is another data processing unit that has an access to $\by^t$ and $\bu^{t-1}$, and generates a control input $\bu_t$ at time step $t$. 
We denote by $\pi_c$ the space of the controller's policies.
We are interested in jointly designing the sensor's and the controller's decision-making policies to solve the following optimization problem:
\begin{align}
\label{mainprob}
\min_{\pi_s\times \pi_c}  J_{\text{cont}}+J_{\text{info}}
\end{align}
\vspace{-2ex}
\begin{align*}
 J_{\text{cont}}&\triangleq \sum_{t=1}^T \frac{1}{2}\mathbb{E}\left( \|\bx_{t+1}\|_{Q_t}^2+\|\bu_t\|_{R_t}^2  \right) \\
J_{\text{info}}&\triangleq \sum_{t=1}^T \gamma_t  I(\bx_t; \by_t| \by^{t-1},\bu^{t-1}).
\end{align*}
We assume that $Q_t \succeq 0$ and $R_t \succ 0$ for every $t=1,\cdots,T$.
The term $I(\bx_t; \by_t| \by^{t-1}, \bu^{t-1})$ denotes the conditional mutual information \cite{CoverThomas}, and $\gamma_t$ is a positive scalar for every $t=1,\cdots, T$\footnote{Suppose $\gamma_t=1 \forall t=1,\cdots, T$. Under the sensor-control architecture we propose in Section \ref{secsummary}, it can be shown that $J_{\text{info}}=I(\bx^T\rightarrow \by^T)$, where the right hand side is known as the \emph{directed information} \cite{massey2005conservation}.}. 
We call (\ref{mainprob}) the \emph{information-regularized LQG control problem}.
\begin{figure}[t]
    \centering
    \includegraphics[width=0.65\columnwidth]{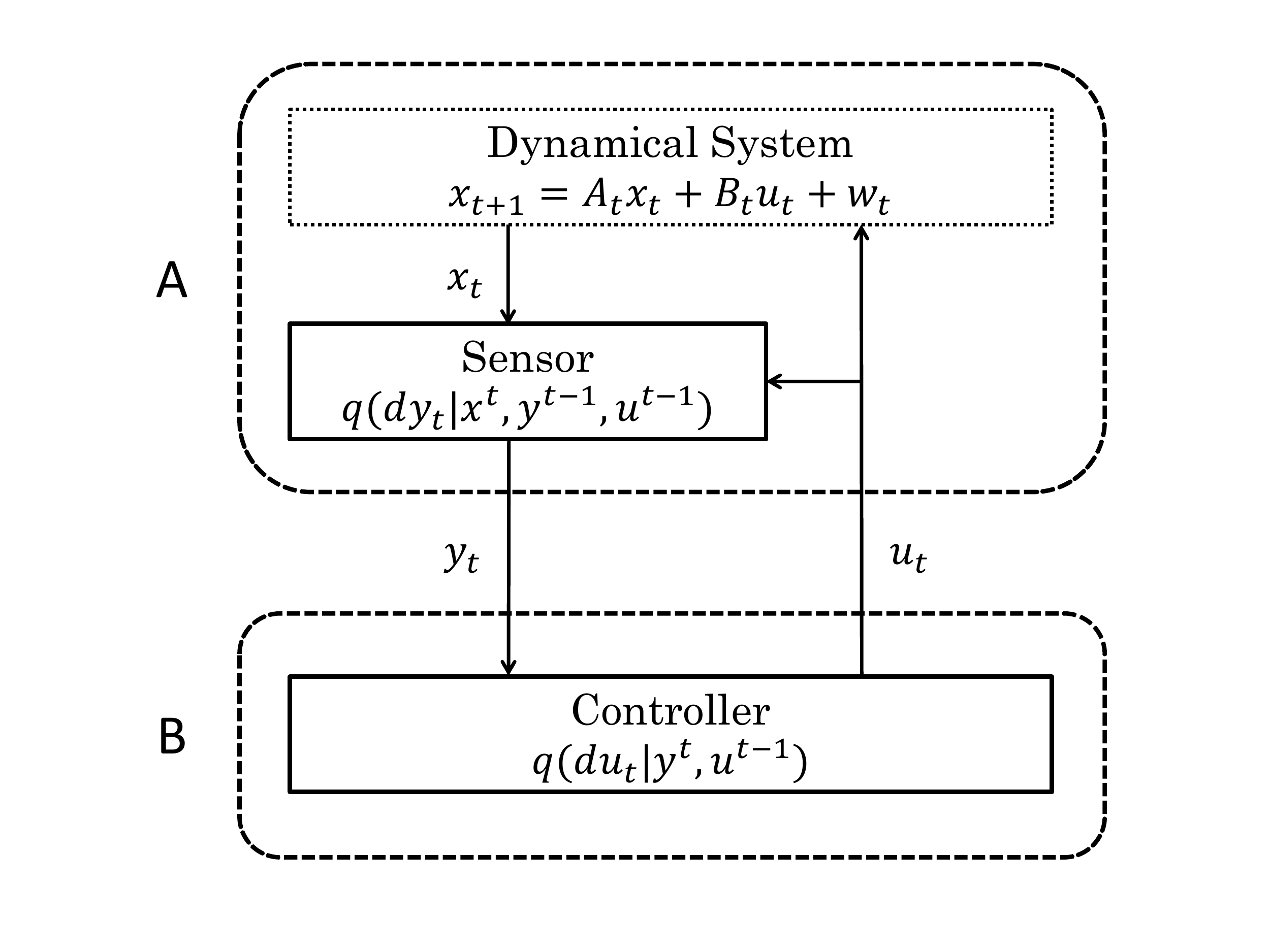}
    \caption{Information-regularized LQG control problem}
    \vspace{-2ex}
    \label{fig:infoLQG}
\end{figure}

In the standard LQG control theory, the sensing policy is typically assumed to be
\begin{equation}
\by_t = C_t \bx_t + \bv_t, \; \bv_t \sim \mathcal{N}(0, V_t) \label{senseeq}
\end{equation}
where $\bv_t$ is a white Gaussian stochastic process and the matrices $\{C_t, V_t\}_{t=1}^T$ are given. 
Due to the well-known separation principle, the optimal controller policy for the standard LQG control problem $\min_{\pi_c} J_{\text{cont}}$ can be found by solving forward and backward Riccati recursions. In (\ref{mainprob}), in contrast, we do not assume (\ref{senseeq}), and allow sensors to be any causal data collecting mechanism in $\pi_s$.
However, $\min_{\pi_s \times \pi_c} J_{\text{cont}}$ is an uninteresting problem, since trivially $\by_t=\bx_t$ (perfect observation) together with the linear-quadratic regulator (LQR) is optimal.
To exclude this trivial solution, we aim at minimizing $J_{\text{cont}}+J_{\text{info}}$ as in (\ref{mainprob}), which amounts to charging the cost $\gamma_t$ for every \emph{bit} of innovative information collected by the sensor at time step $t$. Notice that full observation $\by_t=\bx_t$ results in $J_{\text{info}}=+\infty$.

Although we are not aware of a complete solution to (\ref{mainprob}), we here provide an SDP-based algorithm to construct an optimal \emph{linear} sensor-controller joint policy.
This result turns out to be an extension of an SDP-based algorithm for the \emph{sequential rate-distortion problem} proposed in \cite{1411.7632}.

\section{Applications and Related Work}
\label{secapp}
In this section, we briefly summarize connections between 
the information-regularized LQG control problem  (\ref{mainprob})  and  related  work in the control, information theory, robotics, social science, and economics literature.

\subsection{Control over a communication channel}
In Fig. \ref{fig:infoLQG}, suppose that agents A and B are geographically separated, and the communication channel from A to B is band-limited. Suppose that agents A and B are in collaboration to design the sensor and controller blocks. 
What kind of data should then a sensor collect and transmit, so that B can generate a  satisfactory control signal in a real-time manner?

Feedback control over noisy channels has been a popular research topic in the past two decades. Most of the early contributions focus on stabilization of unstable dynamical systems using feeback control over band-limited communication channels.
A very partial list of papers in this context is \cite{delchamps1990stabilizing,brockett2000quantized,elia2001stabilization,nair2004stabilizability,tsumura2003stabilizability,tatikonda2004control,sahai2006}.
This research direction naturally leads to trade-off studies between the achievable control performance and the required capacity of the sensor-controller communication channel. 
If the communication rate is finite, larger block length (achieving high resolution) is not necessarily preferred since the resulting delay leads to the loss of control performance \cite{borkar1997lqg}.
LQG control performance subject to capacity constraints is considered in \cite{charalambous2008}, where a certain ``separation principle" between control design and communication design is reported. The authors of \cite{freudenberg2011} consider a fundamental performance limitation of the finite horizon minimum-variance control (MVC) over noisy communication channels in the LQG regime.
More comprehensive literature surveys on control designs over communication channels are available in 
\cite{nair2007,hespanha2007survey,baillieul2007,yuksel2013stochastic}.

However, the majority of the existing work in this context assume sensor models and/or channel models \emph{a priori}, and are different from (\ref{mainprob}). A few exceptions include \cite{bansal1989} and \cite{miller1997}, where sensor-controller joint design problems are considered. However, these works are concerned with sensor power constraints rather than information constraints, and are different from (\ref{mainprob}).
Our problem formulation (\ref{mainprob}) falls into a general class of sensor optimization problems considered in \cite{yuksel2012optimization}, where several results are derived regarding the convexity of the problem and the existence of an optimal solution under different choices of topologies in the space of sensors. However, no structural results on specific problems appear there.

\subsection{Bounded rationality}
Broadly, the term \emph{bounded rationality} is used to refer to the limited ability of decision makers (human or robot) to acquire and process information. 
The \emph{rational inattention} model introduced by \cite{sims2003implications} in the economics literature characterizes bounded rationality using the idea of Shannon's channel capacity.
Inspired by this model, recently \cite{shafieepoorfard2013rational} considered an information-constrained LQG control problem, which is similar to  (\ref{mainprob}). 
In this paper, we remove the somewhat restrictive 
 assumptions made in \cite{shafieepoorfard2013rational}, including that a controller there is a time invariant function of the current state only. Furthermore, our SDP-based approach is powerful in handling multi-dimensional systems, while \cite{shafieepoorfard2013rational} is currently restricted to scalar systems.

\subsection{Privacy-preserving control}
In Fig. \ref{fig:infoLQG}, suppose that agent A can privately observe its internal state $\bx_t$, and that $\bx_t$ must be controlled by an external agent B through control input $\bu_t$.
At every time step, a message $\by_t$ containing information about the current state $\bx_t$ is created by the agent A and is sent to the agent B, so that B can compute desirable control inputs. However, sending $\by_t=\bx_t$ may not be desirable for a privacy-aware agent A, since this means a complete loss of privacy. What is then the optimal message $\by_t$?

Suppose that the loss of privacy caused by disclosing $\by_t$ at time step $t$ is quantified by the conditional mutual information $I(\bx_t;\by_t|\by^{t-1}, \bu^{t-1})$.
Conditioning on $\by^{t-1}$ and $\bu^{t-1}$ reflects the fact that agent B knows a realization of these random variables by the time he receives a new message $\by_t$. (Similar quantities are used to evaluate privacy in wiretap channel problems \cite{wyner1975wire}, as well as in more recent database literature \cite{Sankar2013}\cite{Makhdoumi14}.)
Introducing the ``price of privacy" $\gamma_t$, the optimal privacy-preserving control problem can be formulated as (\ref{mainprob}).
In contrast, \cite{le2014differentially} employs \emph{differential privacy} as a privacy measure in dynamic state estimation problems.

\section{Problem Formulation}
\subsection{Information-regularized LQG control problem}
In this paper, both the sensor's and controller's policies are modeled by Borel-measurable stochastic kernels. Set $\mathcal{X}=\mathbb{R}^n$ and $\mathcal{Y}=\mathbb{R}^m$ and let $\mathcal{B_X}$ and $\mathcal{B_Y}$ be the Borel $\sigma$-algebras on $\mathcal{X}$ and $\mathcal{Y}$ respectively, with respect to the usual topology.
\begin{definition}
A \emph{Borel-measurable stochastic kernel} from $(\mathcal{X}\!, \mathcal{B_X})$ to $(\mathcal{Y}\!, \mathcal{B_Y})$ is a map $q(\cdot|\cdot)\!: \mathcal{B_Y} \!\times\! \mathcal{X} \!\rightarrow\! [0,1]$ such that
\begin{itemize}[leftmargin=*]
\item $q(\cdot|x)$ is a probability measure on $(\mathcal{Y}, \mathcal{B_Y})$ for every $x\!\in\! \mathcal{X}$.
\item $q(E|\cdot)$ is $\mathcal{B_X}$-measurable for every $E\in \mathcal{B_Y}$.
\end{itemize}
\end{definition}
A Borel-measurable stochastic kernel from $(\mathcal{X}, \mathcal{B_X})$ to $(\mathcal{Y}, \mathcal{B_Y})$ will be simply referred to as a stochastic kernel from $\mathcal{X}$ to $\mathcal{Y}$, and denoted by $q(dy|x)$.
The space of stochastic kernels from $\mathcal{X}$ to $\mathcal{Y}$ is denoted by $\mathcal{Q}_{\by|\bx}$.

The sensor's policy at time $t$ is a stochastic kernel from $\mathcal{X}^t\times \mathcal{Y}^{t-1}\times \mathcal{U}^{t-1}$ to $\mathcal{Y}_t$. The controller's policy at time $t$ is a stochastic kernel from $\mathcal{Y}^t\times \mathcal{U}^{t-1}$ to $\mathcal{U}_t$. Using the notation above, the policy spaces  
 $\pi_s$ and $\pi_c$  are formally defined by
\[
\pi_s=\prod_{t=1}^T  \mathcal{Q}_{\by_t|\bx^t,\by^{t-1},\bu^{t-1}} \;\;\text{ and }\;\;
\pi_c=\prod_{t=1}^T \mathcal{Q}_{\bu_t|\by^t,\bu^{t-1}}.
\]
Then, (\ref{mainprob}) is an optimization problem over the sequences of stochastic kernels $\{q(dy_t|x^t,y^{t-1},u^{t-1})\}_{t=1}^T \in \pi_s$ and  $\{q(du_t|y^{t},u^{t-1})\}_{t=1}^T \in \pi_c$.
Once an element in $\pi_s\times \pi_c$ is picked, then a joint probability measure $p(dx^T,dy^T,du^T)$ over $\mathcal{X}^T\times \mathcal{Y}^T\times \mathcal{Z}^T$ is uniquely determined (see Proposition~7.28 in \cite{bertsekas1978stochastic}).

\subsection{Restricted problem}
To the best of the authors' knowledge, little is known about the structure of the optimal solution to (\ref{mainprob}). Namely, it is currently unknown whether there exists a jointly \emph{linear} policy in $\pi_s \times \pi_c$ that attains optimality in (\ref{mainprob})\footnote{Our problem is different from the optimal LQG control over Gaussian channels, where a linear encoder-controller pair is optimal (e.g.,  \cite{yuksel2013stochastic} Ch. 11).}.
Hence, in this paper, we focus on a restricted problem in which sensor's policy is restricted to the form  (\ref{senseeq}).
That is, we consider
\begin{align}
\label{simpleprob}
\min_{\pi_s^{\text{lin}}\times \pi_c}  J_{\text{cont}}+J_{\text{info}}
\end{align}
where $\pi_s^{\text{lin}}$ is the space of sequences of stochastic kernels $\{q(dy_t|x_t)\}_{t=1}^T$, which can be realized by a linear sensor equation (\ref{senseeq}) with some $C_t$ and $V_t$ to be determined.
We tackle this problem by applying an SDP-based solution to the \emph{sequential rate-distortion} (SRD) problem obtained in \cite{1411.7632}.
Based on the existence of a linear optimal solution to the Gaussian SRD problem (as shown in \cite{TatikondaThesis}), we will show that (\ref{simpleprob}) has a jointly linear optimal solution.

\section{Summary of the Result}
\label{secsummary}

In this section, we provide a complete solution to the restricted information-regularized LQG control problem (\ref{simpleprob}).
Specifically, we claim that the following numerical procedure allows us to explicitly construct the optimal stochastic kernels $\{q(dy_t|x_t)\}_{t=1}^T \in \pi_s^{\text{lin}}$ and $\{q(du_t|y^t, u^{t-1})\}_{t=1}^T\in \pi_c$ for (\ref{simpleprob}).

{\bf Step~1.} (Controller design) Compute a backward Riccati recursion.
\begin{subequations}
\begin{align}
S_t&=\begin{cases} Q_t & \text{ if } t=T \\ Q_t+N_{t+1} & \text{ if } t=1,\cdots, T-1 \end{cases} \\
M_t&=B_t^\top S_t B_t + R_t \\
N_t&=A_t^\top (S_t-S_t B_t M_t^{-1} B_t^\top S_t) A_t \\
K_t&= -M_t^{-1} B_t^\top S_t A_t \\
\Theta_t&= K_t^\top M_t K_t 
\end{align}
\end{subequations}
The matrix $S_t$ is commonly understood in the LQR theory as the ``cost-to-go" function, while $K_t$ is the optimal control gain. The auxiliary parameter $\Theta_t$ will be used in Step 2.

{\bf Step~2.} (Covariance scheduling) Solve a max-det problem with respect to $\{P_{t|t}, \Pi_t\}_{t=1}^T$ subject to the LMI constraints:
\begin{subequations}
\label{optprob3}
\begin{align}
\min & \;\; \sum_{t=1}^T \left(\frac{1}{2}\text{Tr}(\Theta_t P_{t|t})-\frac{\gamma_t}{2}\log\det \Pi_t \right)+C \\
\text{s.t.} & \;\; \Pi_t \succ  0, \;\; t=1, \cdots, T  \label{optprob3Pi}\\
& \;\; P_{t+1|t+1}\preceq A_t P_{t|t}A_t^\top +W_t, \;\; t\!=\!1,\cdots\!, T\!-\!1 \\
& \;\; P_{1|1}\preceq P_{1|0}, P_{T|T}=\Pi_T \\
&  \left[\!\! \begin{array}{cc}P_{t|t}\!-\!\Pi_t\!\!\! &\!\! P_{t|t}A_t^\top \\
A_tP_{t|t} \!\!\!&\!\! W_t\!+\!A_t P_{t|t}A_t^\top  \end{array}\!\!\right]\! \succeq\! 0, \;\; t\!=\!1,\cdots\!, T\!-\!1 
\end{align}
\end{subequations} 
where $C$ is a constant\footnote{
The constant is given by
\begin{align*}
C=&\sum_{t=1}^{T-1} \!\left( \frac{\gamma_t n_t}{2}\log \frac{\gamma_{t+1}}{\gamma_t} 
+\!\frac{\gamma_t}{2}\log\det W_t\right) \!\!+\!\frac{\gamma_1}{2}\log\det P_{1|0} \\ 
&+\frac{1}{2}\text{Tr}(N_1P_{1|0})+\frac{1}{2}\sum_{t=1}^T\text{Tr}(W_tS_t).
\end{align*}}. Due to the boundedness of the feasible set, (\ref{optprob3}) has an optimal solution\footnote{One can replace (\ref{optprob3Pi}) with $\Pi_t\succeq \epsilon I$ without altering the result. This conversion makes the feasible set compact and thus the Weierstrass theorem can be used.}.

{\bf Step~3.} (Sensor design) Set $r_t=\text{rank}(P_{t|t}^{-1}-P_{t|t-1}^{-1})$ for every $t=1,\cdots, T$, where 
\[
P_{t|t-1}\triangleq A_{t-1}P_{t-1|t-1}A_{t-1}^\top + W_{t-1}, t=2,\cdots, T.
\]
 Choose matrices $C_t\in \mathbb{R}^{r_t\times n_t}$ and $V_t\in \mathbb{S}_{++}^{r_t}$ so that they satisfy
\begin{equation}
\label{cvconst}
 C_t^\top V_t^{-1}C_t= P_{t|t}^{-1}-P_{t|t-1}^{-1}
\end{equation}
for $t=1,\cdots, T$. For instance, the singular value decomposition can be used. In particular, in case of $r_t=0$, $C_t$ and $V_t$ are considered to be null (zero dimensional) matrices.

{\bf Step~4.} (Filter design) Determine the Kalman gains by
\begin{equation}
\label{eqkalmangain}
L_t=P_{t|t-1}C_t^\top (C_t P_{t|t-1}C_t^\top + V_t)^{-1}.
\end{equation}
If $r_t=0$, $L_t$ is a null matrix.

{\bf Step~5.} (Policy construction) Using $\{C_t, V_t, L_t, K_t\}_{t=1}^T$ obtained above, define the sensor's policy $\{q(dy_t|x_t)\}_{t=1}^T \in \pi_s^{\text{lin}}$ by equation (\ref{senseeq}). When $r_t=0$, the optimal dimension of the sensing vector $\by_t$ is zero, meaning that no sensing is the optimal strategy.  On the other hand, define a controller's policy $\{q(du_t|y^t, u^{t-1})\}_{t=1}^T\in \pi_c$ by the certainty equivalence controller $\bu_t = K_t \hat{\bx}_t$ where $\hat{\bx}_t=\mathbb{E}(\bx_t|\by^t, \bu^{t-1})$ is obtained by the standard Kalman filter
\begin{subequations}
\label{eqkalmanfilter}
\begin{align}
&\hat{\bx}_t=\hat{\bx}_{t|t-1}+L_t(\by_t-C_t\hat{\bx}_{t|t-1}) \label{eqkalmanfilter1}\\
&\hat{\bx}_{t+1|t}=A_t\hat{\bx}_t+B_t \bu_t. \label{eqkalmanfilter2}
\end{align}
\end{subequations}
When $r_t=0$, (\ref{eqkalmanfilter1}) is simply replaced by $\hat{\bx}_t=\hat{\bx}_{t|t-1}$.
\newpage
\begin{theorem}
There exists a joint sensor-controller policy in $\pi_s^{\text{lin}}\times \pi_c$ that attains optimality in (\ref{simpleprob}). The optimal value of (\ref{simpleprob}) coincides with the optimal value of the max-det problem~(\ref{optprob3}). Furthermore, an optimal policy can be constructed by the Steps 1-5.
\end{theorem}

\section{Derivation of the Main Result}
\label{secderivation}
We first show that, once the sensor's policy $\{q_{\by_t|\bx_t}\}_{t=1}^T \in \pi_s^{\text{lin}}$ is fixed, then $J_{\text{info}}$ does not depend on the choice of controller's policy. This observation allows us to rewrite (\ref{mainprob}) as
\begin{equation}
\label{eqstackelberg}
\min_{\pi_s^{\text{lin}}} \left(J_{\text{info}} +\min_{\pi_c} J_{\text{cont}}\right). 
\end{equation}
Then, we interpret (\ref{eqstackelberg}) as a two-player Stackelberg game (see, e.g., \cite{tamer1999dynamic}) in which the sensor agent (agent A in Fig. \ref{fig:infoLQG}) is the leader and the controller agent (agent B) is the follower. 
If the sensor's policy is given, the controller's best response can be explicitly found by solving a stochastic optimal control problem $\min_{\pi_c} J_{\text{cont}}$.
With an explicit expression of $\min_{\pi_c} J_{\text{cont}}$, we show that the outer optimization problem in (\ref{eqstackelberg}) over $\pi_s^{\text{lin}}$ becomes the \emph{sequential rate-distortion problem} \cite{TatikondaThesis}, whose optimal solution can be constructed by solving an SDP problem \cite{1411.7632}.

Fix a joint sensor-controller policy in $\pi_s^{\text{lin}}\times \pi_c$ and let $p(dx^T,dy^T,du^T)$ be the resulting joint probability measure. Let $p(dx_t|y^{t-1},u^{t-1})$ and $p(dx_t|y^{t},u^{t-1})$ be probability measures obtained by conditioning and marginalizing $p(dx^T,dy^T,du^T)$. It follows from the standard Kalman filtering theory that
\begin{align*}
p(dx_t|y^{t-1},u^{t-1}) &\sim \mathcal{N}(\hat{\bx}_{t|t-1}, P_{t|t-1}) \\
p(dx_t|y^t,u^{t-1}) &\sim \mathcal{N}(\hat{\bx}_t, P_{t|t}) 
\end{align*}
where $\{P_{t|t}\}$ and $\{P_{t|t-1}\}$ satisfy 
\begin{subequations}
\label{kalmanrec}
\begin{align}
 P_{t|t-1}&=A_{t-1}P_{t-1|t-1}A_{t-1}^\top + W_{t-1} \\
 P_{t|t}&=(P_{t|t-1}^{-1}+C_t^\top V_t^{-1}C_t)^{-1},
\end{align}
\end{subequations}
while $\hat{\bx}_t$ and $\hat{\bx}_{t|t-1}$ are recursively obtained by (\ref{eqkalmanfilter}).
Using matrices $\{P_{t|t}\}$ and $\{P_{t|t-1}\}$, the mutual information terms can be explicitly written as
\begin{align}
I(\bx_t;\by_t|\by^{t-1},\bu^{t-1})&=h(\bx_t|\by^{t-1},\bu^{t-1})-h(\bx_t|\by^t,\bu^{t-1}) \nonumber \\
&=\frac{1}{2}\log\det P_{t|t-1}\!-\!\frac{1}{2}\log\det P_{t|t}. \nonumber
\end{align}
Therefore,
\begin{align*}
J_{\text{info}}&\triangleq \sum_{t=1}^T \gamma_t I(\bx_t;\by_t|\by^{t-1},\bu^{t-1}) \\
&=\sum_{t=1}^{T-1}\left(\frac{\gamma_{t+1}}{2}\log\det P_{t+1|t}-\frac{\gamma_t}{2}\log\det P_{t|t}\right) \\
&\hspace{10ex}+\frac{\gamma_1}{2}\log\det P_{1|0}-\frac{\gamma_T}{2}\log\det P_{T|T}.
\end{align*}
In particular, this result clearly shows that the mutual information terms are \emph{control-independent}, since they are completely determined by (\ref{kalmanrec}) once the sequence of matrices $\{C_t,V_t\}_{t=1}^T$ is fixed. This observation justifies the equivalence between (\ref{simpleprob}) and (\ref{eqstackelberg}).

Next, let us focus on the stochastic optimal control problem $\min_{\pi_c} J_{\text{cont}}$, whose solution is well understood.
\begin{lemma}\label{claim1}
For every fixed $\{q_{\by_t|\bx_t}\}_{t=1}^T \in \pi_s^{\text{lin}}$, 
the certainty equivalence controller $\bu_t=K_t\hat{\bx}_t$ where $\hat{\bx}_t=\mathbb{E}(\bx_t|\by^t,\bu^{t-1})$ is an optimizer of $\min_{\pi_c} J_{\text{cont}}$. 
Moreover,
\[
\min_{\pi_c} J_{\text{cont}}\!=\!\frac{1}{2}\text{Tr}(N_1\!P_{1|0})\!+\!\frac{1}{2}\!\sum_{k=1}^T\!\!\left(\text{Tr}(W_kS_k)\!+\!\text{Tr}(\Theta_kP_{k|k}) \right)\!.\]
\end{lemma}
\begin{proof}
This is a standard result and can be shown by dynamic programming.
A proof is provided in Appendix.
\end{proof}

Combining the results so far, we have shown that
\begin{align}
& J_{\text{info}}+ \min_{\pi_c} J_{\text{cont}} \nonumber \\
&=\!\!\sum_{t=1}^{T-1}\!\left(\!\frac{1}{2}\text{Tr}(\Theta_t P_{t|t})\!+\!\frac{\gamma_{t+1}}{2}\log\det P_{t+1|t}\!-\!\frac{\gamma_t}{2}\log\det P_{t|t}\! \right) \nonumber \\
&\;\;\;\;\;+\frac{1}{2}\text{Tr}(\Theta_T P_{T|T})-\frac{\gamma_T}{2}\log\det P_{T|T}+c \label{bestresponse}
\end{align}
where $c\!=\!\frac{1}{2}\text{Tr}(N_1P_{1|0})+\frac{\gamma_1}{2}\log\det P_{1|0}+\frac{1}{2}\sum_{t=1}^T \!\! \text{Tr}(W_tS_t)$
 is a constant. The expression (\ref{bestresponse}) is the cost of the original problem (\ref{mainprob}) when the sensor model $\{C_t, V_t\}_{t=1}^T$ is fixed and the controller agent (Stackelberg follower) reacts with the best response. Notice that  (\ref{bestresponse}) is a function of the sequence $\{C_t, V_t\}_{t=1}^T$, since the matrices $P_{t|t}$ and $P_{t+1|t}$ are determined by  (\ref{kalmanrec}). 
 
 Now we have formulated a problem for the sensor agent (Stackelberg leader). Namely, the sensor agent needs to find the optimal sequence of matrices $\{C_t, V_t\}_{t=1}^T$ (as well as their dimensions) that minimizes (\ref{bestresponse}). Next, we show that this can be done very efficiently by solving a semidefinite programming problem.

Let us first focus on the quantity
\begin{equation}
\label{gammap}
\frac{\gamma_{t+1}}{2}\log\det P_{t+1|t}-\frac{\gamma_t}{2}\log\det P_{t|t}.
\end{equation}
Introducing $\tilde{A}_t=\sqrt{\frac{\gamma_{t+1}}{\gamma_t}}A_t$ and $\tilde{W}_t=\frac{\gamma_{t+1}}{\gamma_t}W_t$,

\begin{subequations}
\begin{align}
\frac{2}{\gamma_t} &\times  \text{(\ref{gammap})} 
= \log\det (\tilde{A}_tP_{t|t}\tilde{A}_t^\top+\tilde{W}_t)-\log\det P_{t|t} \nonumber \\
&=\log\det \tilde{W}_t-\log\det(P_{t|t}^{-1}+A_t^\top W_t^{-1}A_t)^{-1} \label{sdpderivation1_2} \\
&=\min \;\; \log\det W_t\!-\!\log\det \Pi_t\!+\!n_t\log\left(\!\frac{\gamma_{t+1}}{\gamma_t}\!\right) \label{sdpderivation2} \\
& \hspace{3ex} \text{s.t.} \hspace{3ex} 0\prec \Pi_t \preceq (P_{t|t}^{-1}+A_t^\top W_t^{-1}A_t)^{-1} \nonumber \\
&=\min \;\; \log\det W_t\!-\!\log\det \Pi_t\!+\!n_t\log\left(\!\frac{\gamma_{t+1}}{\gamma_t}\!\right) \label{sdpderivation3} \\
& \hspace{3ex} \text{s.t.} \hspace{2ex} \left[\!\begin{array}{cc}
P_{t|t}\!-\!\Pi_t \!&\! P_{t|t}A_t^\top \\
A_t P_{t|t} \!&\! A_tP_{t|t}A_t^\top \!+\!W_t
\end{array}\!\right] \succeq 0, \Pi_t \succ 0 \nonumber 
\end{align}
\end{subequations}
We have used Sylvester's determinant theorem in step~(\ref{sdpderivation1_2}).  The quantity (\ref{sdpderivation1_2}) is equal to the optimal value of a constrained optimization problem (\ref{sdpderivation2}) with decision variables $P_{t|t}$ and $\Pi_t$, and this rewriting is possible because of the monotonicity of the determinant function. In  (\ref{sdpderivation3}), the constraint $\Pi_t \preceq (P_{t|t}^{-1}+A_t^\top W_t^{-1}A_t)^{-1}$ is rewritten using the Schur complement formula. The final expression  (\ref{sdpderivation3}) is particularly useful, since this is a max-det problem subject to linear matrix inequality (LMI) constraints.

Applying the discussion above to every $t=1,\cdots, T-1$, and introducing $\Pi_T=P_{T|T}$ for notational convenience, it follows from (\ref{bestresponse}) that the optimal $J$ is equal to the value of the following optimization problem with respect to the decision variables $\{P_{t|t}, \Pi_t, C_t, V_t\}_{t=1}^T$:
\begin{align*}
\min & \sum_{t=1}^T \left(\frac{1}{2}\text{Tr}(\Theta_t P_{t|t})-\frac{\gamma_t}{2}\log\det \Pi_t\right)+C \\
\text{s.t.} & \left[\!\begin{array}{cc}
P_{t|t}\!-\!\Pi_t \!&\! P_{t|t}A_t^\top \\
A_t P_{t|t} \!&\! A_tP_{t|t}A_t^\top \!+\!W_t
\end{array}\!\right] \succeq 0, t\!=\!1,\cdots\!, T\!-\!1 \\
& \Pi_t \succ 0, t=1,\cdots T \\
& \Pi_T=P_{T|T} \\
& P_{1|1}^{-1}=P_{1|0}^{-1}+C_1^\top V_1^{-1} C_1 \\
& P_{t|t}^{-1}\!=\!(A_{t\!-\!1}P_{t\!-\!1|t\!-\!1}A_{t\!-\!1}^\top+\!W_t)^{-1}\!+C_t^\top V_t^{-1} C_t\\
& \hspace{35ex} t=2,\cdots, T.
\end{align*}
The last two constraints are obtained by eliminating $P_{t|t-1}$ from (\ref{kalmanrec}). These equality constraints themselves are difficult to handle, but can be replaced by the inequality constraints
\begin{align*}
&0\prec P_{1|1} \preceq P_{1|0} \\
&0\prec P_{t|t} \preceq A_{t-1}P_{t-1|t-1}A_{t-1}^\top+W_{t-1}.
\end{align*}
These replacements eliminate the variables $\{C_t, V_t\}_{t=1}^T$, and convert the above optimization problem into an alternative problem with respect to $\{P_{t|t}, \Pi_t\}_{t=1}^T$ only, as shown in (\ref{optprob3}).  The eliminated variables $\{C_t, V_t\}_{t=1}^T$ can be easily reconstructed by (\ref{cvconst}).

Solving (\ref{optprob3}) allows us to optimally schedule the sequence of covariance matrices. The optimal covariance sequence can be attained by the Kalman filter (\ref{eqkalmanfilter}).

\begin{figure*}[t]
\centering
\begin{equation}
\label{equationmotion}
\left[\!\!\begin{array}{c}
\dot{\phi} \\ \dot{\theta} \\ \dot{\psi} \\ \dot{\omega}_\phi \\ \dot{\omega}_\theta \\ \dot{\omega}_\psi
 \end{array}\!\!\right]\!\!=\!\!
\left[\!\!\begin{array}{cccccc} 
0&0&0&1&0&0 \\
0&0&0&0&1&0 \\
0&0&0&0&0&1 \\
\!\!-4\omega_0^2 \sigma_x\!\!\! &0&0&0&0& \!\!\!\omega_0(1\!-\!\sigma_x) \\
0& \!\!\!3\omega_0^2 \sigma_y\!\!\!& 0&0&0&0 \\
0&0&\!\!\!\omega_0^2 \sigma_z\!\! & \!\! -\omega_0(1\!+\!\sigma_z)\!\!\! &0&1 
\end{array}\!\!\right]\!\!\!
\left[\!\!\begin{array}{c}
\phi \\ \theta \\ \psi \\ \omega_\phi \\ \omega_\theta \\ \omega_\psi
 \end{array}\!\!\right]\!\!+\!\!
\left[\!\!\begin{array}{ccc} 
0&0&0 \\
0&0&0 \\
0&0&0 \\
0&\!\!\!b_z(t)/I_x\!\!\! &\!\!\!-b_y(t)/I_x \\
-b_z(t)/I_y\!\!\!&0&\!\!b_x(t)/I_y\\
b_y(t)/I_z\!\!\!&\!\!\!-b_x(t)/I_z\!\!\!&0 
\end{array}\!\!\right]\!\!\!
\left[\!\begin{array}{c}u_x \\ u_y \\ u_z \end{array}\!\right]\!\!+\!\!
\left[\!\begin{array}{c}
w_\phi \\ w_\theta \\ w_\psi \\  n_\phi \\ n_\theta \\ n_\psi\end{array}\!\right]
\end{equation}
\vspace{-.07in}
\hrule
\end{figure*}
\begin{figure}[t]
    \centering
    \includegraphics[width=0.6\columnwidth]{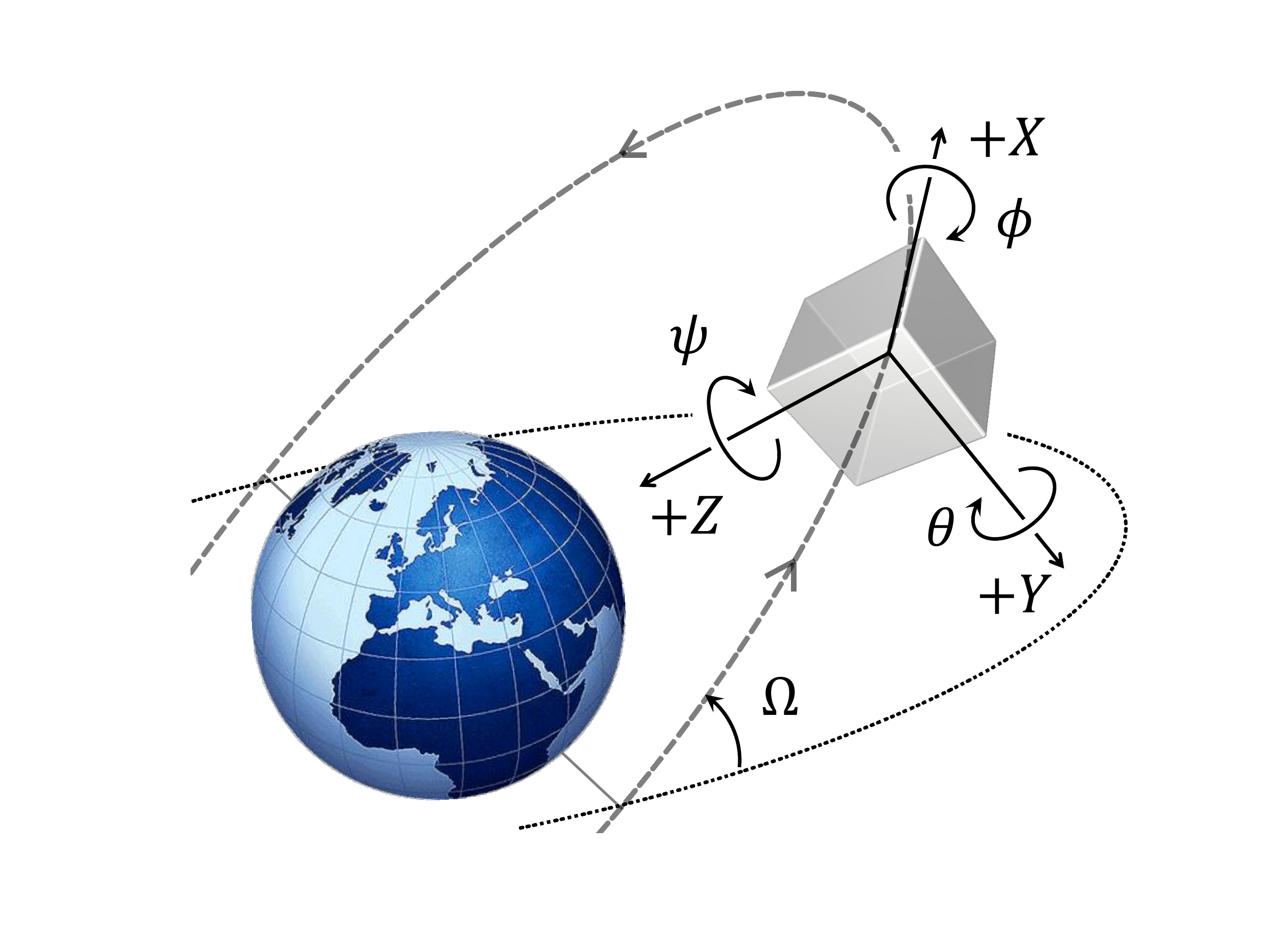}
    \caption{Orbital coordinate and desired attitude of a nadir-pointing satellite.}
    \vspace{-2ex}
    \label{fig:nadir}
\end{figure}
\begin{figure}[h]
    \centering
    \includegraphics[width=\columnwidth]{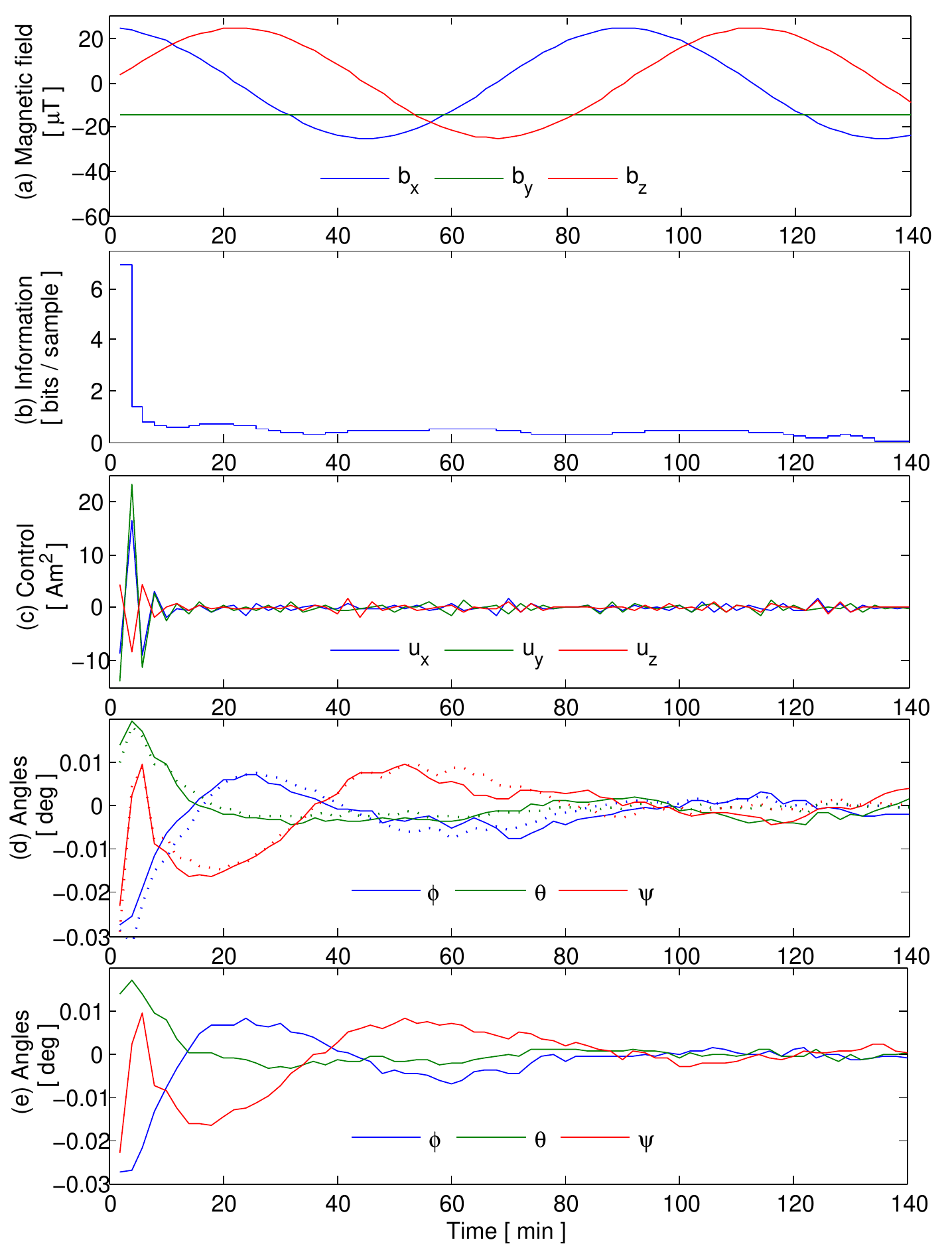}
    \caption{Simulation (Sampling period = 2 min)}
    \vspace{-2ex}
    \label{fig:satellite}
\end{figure}

\section{Example}
In this section, we design an attitude control law for a nadir-pointing spacecraft (Fig. \ref{fig:nadir}) using magnetic torquers.
Small deviations of the body coordinates from the orbital coordinates are measured by angles $\phi, \theta$ and $\psi$, and their dynamics is modeled by a linearized equation of motion (\ref{equationmotion}) borrowed from \cite{psiaki2001magnetic}.

Here, $\omega_0$ is the orbital rate ($2\pi$ [rad] / $90$ [min]), $\text{diag}(I_x,I_y,I_z)$ is the moment of inertia of the  spacecraft, and $\sigma_x=\frac{I_y-I_z}{I_x}$, $\sigma_y=\frac{I_z-I_x}{I_y}$, $\sigma_z=\frac{I_x-I_y}{I_z}$.  The Earth's magnetic field vector $(b_x(t),b_y(t),b_z(t))$ in the orbital coordinate is time varying as the position of the spacecraft changes (the simplified model of the magnetic field shown in Fig. \ref{fig:satellite} (a) will be used). $(u_x,u_y,u_z)$ is the control output of the magnetic torquers. Assume that the spacecraft has an attitude sensor that measures $\phi,\theta,\psi,\omega_\phi,\omega_\theta,\omega_\psi$ accurately, but communication between the attitude sensor and the magnetic torquers incurs cost.

The information-regularized LQG control problem (\ref{simpleprob}) is formulated over the planning horizon of 140 minutes after converting (\ref{equationmotion}) into a discrete-time model with sampling period of 2 minutes, and all necessary parameters are appropriately chosen. 
Fig. \ref{fig:satellite} shows a result of a sensor-controller joint strategy obtained by the Steps 1-5.
Fig. \ref{fig:satellite} (b) shows the optimal assignment of the information rate $I(\bx_t;\by_t|\by^{t-1},\bu^{t-1}), t=1,\cdots, 70$. It can be seen that acquiring a lot of information at the beginning is advantageous in this example.
Simulated control actions and deviation from the desired attitude are shown in subfigures (c) and (d). 
The dotted lines in (d) show the estimated deviations calculated by the Kalman filter (\ref{eqkalmanfilter}).
The resulting costs are $J_{\text{cont}}=1.472$ and $J_{\text{info}}=0.245$. 
For comparison purposes, (e) shows the case where the optimal linear quadratic regulator (LQR) is applied with the perfect measurement $\by_t=\bx_t$ with the same noise realizations. In this case, we have $J_{\text{cont}}=1.352$ and $J_{\text{info}}=+\infty$.
It can be seen that the control performance in (d) is not so much worse than (e), even though the information rate required for (d) is drastically smaller.

\section{Discussion and Future Works}

In this paper, we presented an SDP-based optimal joint sensor-controller synthesis for (restricted) information-regularized LQG control problems. Unfortunately, to the best of the author's knowledge, it is not known whether the same architecture remains optimal in the fully general information-regularized LQG control problem (\ref{mainprob}). The technical difficulty here is that once nonlinear sensor policies $\pi_s$ are allowed, the mutual information term $J_{\text{info}}$ is no longer control-independent in general, and the discussion in Section~\ref{secderivation} does not hold.

Finally, the information-regularized LQG control problem considered in this paper can be viewed as a preliminary step towards a unification of the classical LQG control problem and the Gaussian sequential rate-distortion problem (Fig. \ref{fig:sep3}). In the classical LQG control problem where a sensor model is fixed, the estimator-controller separation principle is well-known. On the other hand, if a feedback control is not considered (or controller is fixed), and $J_{\text{cont}}$ is replaced by $\sum_{t=1}^T \mathbb{E}\|\bx_t-\hat{\bx}_t\|^2$, then the problem becomes the Gaussian sequential rate-distortion problem \cite{TatikondaThesis}, and the sensor-estimator separation principle also holds \cite{1411.7632}.

\begin{figure}[t]
    \centering
    \includegraphics[width=0.8\columnwidth]{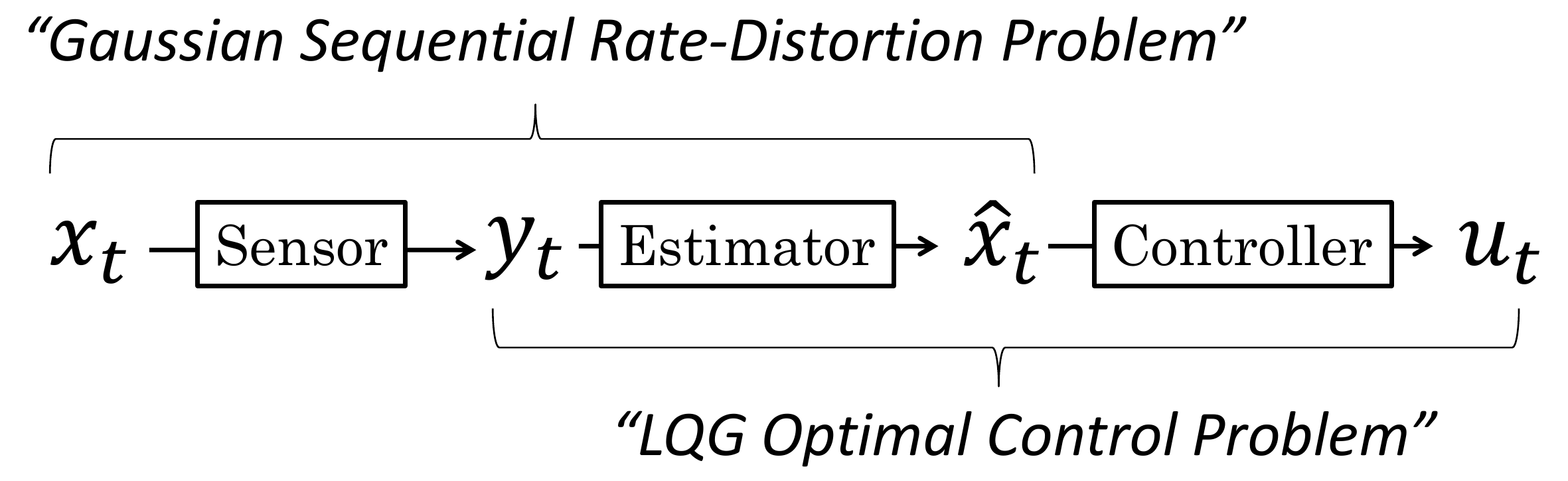}
    \caption{Relationship between sequential rate-distortion problem and LQG optimal control problem.}
    \vspace{-2ex}
    \label{fig:sep3}
\end{figure}


\section*{ACKNOWLEDGMENT}
The authors would like to thank Prof. Sanjoy K. Mitter at MIT for valuable suggestions.

\bibliographystyle{ieeetr}
\bibliography{ACC2015}

\section*{APPENDIX}
We consider $\min_{\pi_c} J_{\text{cont}}$ as a $T$-stage dynamic programming problem.
The state of the system at stage $t$ is a joint probability measure $p(dx^t,dy^t,du^{t-1})$ which is updated by
\begin{align*}
p(dx^{t+1},dy^{t+1},du^t)=& q(dy_{t+1}|x_{t+1})f(dx_{t+1}|x_t,u_t) \\
&\times  q(du_t|y^t,u^{t-1})p(dx^t,dy^t,du^{t-1}).
\end{align*}
Here, the stochastic kernel $f(dx_{t+1}|x_t,u_t)$ is given by (\ref{eqprocess}), while $q(dy_{t+1}|x_{t+1})$ is the sensing policy, which is assumed to be fixed.
The stochastic kernel $q(du_t|y^t,u^{t-1}) \in \mathcal{Q}_{\bu_t|\by^t,\bu^{t-1}}$ is the control variable in this dynamic programming formulation.
The associated Bellman's equation is
\begin{align*} &J_t(p(dx^t,dy^t,du^{t-1}))=\\
&\!\min_{\mathcal{Q}_{\bu_t|\by^t\!\!,\bu^{t\!-\!1}}}\!\! \left\{\! \frac{1}{2}\mathbb{E}\!\left(\|\bx_{t+1}\|_{Q_t}^2 \!\!\!+\!\|\bu_t\|_{R_t}^2\right) \!\!+\!\! J_{t+1}(p(dx^{t+1}\!\!\!,dy^{t+1}\!\!\!,du^t))\!\right\}
\end{align*}
with the boundary condition $J_{T+1}(\cdot)=0$.

\begin{claim}
\label{claima}
For every $t=1,\cdots, T$, the certainty equivalence controller $\bu_t=K_t\hat{\bx}_t$ where $\hat{\bx}\triangleq \mathbb{E}(\bx_t|\by^t,\bu^{t-1})$ is the optimal control policy in
$\mathcal{Q}_{\bu_t|\by^t,\bu^{t-1}}$. Moreover, for every $t=1,\cdots,T$,
\begin{align} 
&J_t(q_{\bx^t,\by^t,\bu^{t-1}})= \nonumber  \\
&\frac{1}{2}  \mathbb{E}\|\bx_t\|_{N_t}^2\!\!+\!\frac{1}{2}\sum_{k=t}^T\left(\text{Tr}(W_kS_k)\!+\! \text{Tr}(\Theta_k P_{k|k})\right). \label{eqjt}
\end{align}
\end{claim}
\begin{proof}
Equation (\ref{eqjt}) holds when $t=T$ as
\begin{align*}
&J_T(q_{\bx^T,\by^T,\bu^{T-1}})  \\
&=\min_{\mathcal{Q}_{\bu_T|\by^T\!\!,\bu^{T\!-\!1}}}\!\frac{1}{2}\mathbb{E}\left(
\|A_t\bx_T\!+\!B_t\bu_T\!+\!\bw_T\|_{Q_T}^2\!+\!\|\bu_T\|_{R_T}^2\right)  \\
&=\min_{\mathcal{Q}_{\bu_T|\by^T\!\!,\bu^{T\!-\!1}}}\!\frac{1}{2}\mathbb{E}\left(
\|\bx_T\|_{N_T}^2\!\!+\!\|\bw_T\|_{Q_T}^2\!\!+\!\|\bu_T\!-\!K_T\bx_T\|_{M_T}^2\right)  \\
&=\frac{1}{2}\mathbb{E}\|\bx_T\|_{N_T}^2+\frac{1}{2}\left(\text{Tr}(W_TQ_T)+\text{Tr}(\Theta_TP_{T|T})\right).
\end{align*}
Notice that in the second expression, $\bu_T$ appears only in $\|\bu_T\!-\!K_T\bx_T\|_{M_T}^2$. By choosing $\bu_T=K_T\mathbb{E}(\bx_T|\by^T)$, this quantity attains its minimum value $\mathbb{E}\|K_T(\bx-\mathbb{E}(\bx_T|\by^T))\|_{M_T}^2=\text{Tr}(\Theta_TP_{T|T})$. So assume (\ref{eqjt}) holds for $t=l+1$. Then
\begin{align*}
&J_l(q_{\bx^l,\by^l,\bu^{l-1}})  \\
&=\min_{\mathcal{Q}_{\bu_T|\by^T\!\!,\bu^{T\!-\!1}}}\!
\left\{\frac{1}{2}\mathbb{E}\left(
\|\bx_{l+1}\|_{Q_l}^2\!+\!\|\bu_l\|_{R_l}^2\right)+\frac{1}{2}\mathbb{E}\|\bx_{l+1}\|_{N_{l+1}}^2  \right. \\
& \hspace{15ex}+\left. \frac{1}{2}\sum_{k=l+1}^T\left(\text{Tr}(W_kS_k)\!+\! \text{Tr}(\Theta_k P_{k|k})\right) \right\} \\
&=\min_{\mathcal{Q}_{\bu_T|\by^T\!\!,\bu^{T\!-\!1}}}\!
\left\{\frac{1}{2}\mathbb{E}\left(
\|\bx_{l+1}\|_{S_l}^2\!+\!\|\bu_l\|_{R_l}^2\right)  \right. \\
& \hspace{15ex}+\left. \frac{1}{2}\sum_{k=l+1}^T\left(\text{Tr}(W_kS_k)\!+\! \text{Tr}(\Theta_k P_{k|k})\right) \right\}  \\
&=\frac{1}{2}\sum_{k=l+1}^T\left(\text{Tr}(W_kS_k)\!+\! \text{Tr}(\Theta_k P_{k|k})\right) \\
& \hspace{5ex}+ \min_{\mathcal{Q}_{\bu_l|\by^l\!\!,\bu^{l\!-\!1}}}\frac{1}{2}\mathbb{E}\left( \|A_l\bx_l+B_l\bu_l+\bw_l\|_{S_l}^2+\|\bu_l\|_{R_l}^2 \right) \\
&=\frac{1}{2}\sum_{k=l+1}^T\left(\text{Tr}(W_kS_k)\!+\! \text{Tr}(\Theta_k P_{k|k})\right) \\
& \hspace{10ex}+ \frac{1}{2}\mathbb{E}\|\bx_l\|_{N_l}^2+\frac{1}{2}\left(\text{Tr}(W_lS_l)\!+\! \text{Tr}(\Theta_l P_{l|l})\right) \\
&=\frac{1}{2}  \mathbb{E}\|\bx_l\|_{N_l}^2+\frac{1}{2}\sum_{k=l}^T\left(\text{Tr}(W_kS_k)+ \text{Tr}(\Theta_k P_{k|k})\right).
\end{align*}
\end{proof}
Noticing $\mathbb{E}\|\bx_1\|_{N_1}^2\!\!=\!\text{Tr}(N_1P_{1|0})$, Lemma \ref{claim1} follows from Claim \ref{claima}.


\end{document}